\documentclass[11pt,a4paper]{amsart}

\usepackage[margin=1in]{geometry}
\usepackage{amsmath,amssymb,amsthm,mathtools,mathrsfs,microtype,charter,cite}
\usepackage{aliascnt}
\usepackage{enumitem,doi}
\usepackage{xcolor,pgfplots}
\allowdisplaybreaks
\pgfplotsset{compat=1.18}

\newtheorem{theorem}{Theorem}[section]
\newaliascnt{proposition}{theorem}
\newtheorem{proposition}[proposition]{Proposition}
\aliascntresetthe{proposition}
\newaliascnt{lemma}{theorem}
\newtheorem{lemma}[lemma]{Lemma}
\aliascntresetthe{lemma}
\newaliascnt{corollary}{theorem}

\aliascntresetthe{corollary}
\theoremstyle{definition}
\newaliascnt{definition}{theorem}

\aliascntresetthe{definition}
\newaliascnt{assumption}{theorem}

\aliascntresetthe{assumption}
\theoremstyle{remark}
\newaliascnt{remark}{theorem}
\newtheorem{remark}[remark]{Remark}
\aliascntresetthe{remark}

\usepackage[capitalise,nameinlink,noabbrev]{cleveref}

\crefname{theorem}{theorem}{theorems}
\Crefname{theorem}{Theorem}{Theorems}
\crefname{proposition}{proposition}{propositions}
\Crefname{proposition}{Proposition}{Propositions}
\crefname{lemma}{lemma}{lemmas}
\Crefname{lemma}{Lemma}{Lemmas}
\crefname{corollary}{corollary}{corollaries}
\Crefname{corollary}{Corollary}{Corollaries}
\crefname{definition}{definition}{definitions}
\Crefname{definition}{Definition}{Definitions}
\crefname{assumption}{assumption}{assumptions}
\Crefname{assumption}{Assumption}{Assumptions}
\crefname{remark}{remark}{remarks}
\Crefname{remark}{Remark}{Remarks}

\newcommand{\R}{\mathbb{R}}
\newcommand{\eps}{\varepsilon}
\newcommand{\CD}[1]{\partial_t^{#1}}
\newcommand{\Hess}{\operatorname{Hess}}
\newcommand{\dd}{\mathrm{d}}
\DeclareMathOperator{\supp}{supp}

\newcommand{\rev}[1]{{\color{black}#1}}
\newenvironment{revision}{\begingroup\color{black}}{\endgroup}

\begin{document}

\title[A semiconvex counterexample to energy monotonicity for time-fractional gradient flows]{\Large A semiconvex counterexample to energy monotonicity \\[.1cm] for time-fractional gradient flows}

\author{Marvin Fritz}

\address{Marvin Fritz,
marvin.fritz@univie.ac.at, Faculty of Mathematics, University of Vienna, Oskar-Morgenstern-Platz 1, 1090 Vienna, Austria
}

\date{}

\maketitle

\mbox{}\vspace{-1cm}

\begin{abstract}
\noindent
For time-fractional gradient flows, natural dissipation statements are often
integrated or memory-modified rather than pointwise differential inequalities
for the original energy. We construct a smooth, compactly supported, globally
semiconvex energy and an absolutely continuous solution on a finite time
interval of a time-fractional gradient flow along which the original energy is
strictly increasing on an explicit subinterval. The construction incorporates
the initial layer directly into the trajectory. After transformation, the curve
and its fractional derivative are polynomial, the curve is a regular immersion,
and the Hessian remains bounded below up to the anchor.
\end{abstract}

\section{Introduction}

\begin{revision}
Time-fractional gradient flows arise when the instantaneous relaxation in a
classical gradient system is replaced by a history-dependent Caputo derivative.
This replacement is natural in subdiffusive models, but it changes the meaning
of energy dissipation: the positive-definiteness of the Abel convolution gives
robust \emph{integrated} stability estimates, while pointwise decay of the
original energy is much more delicate.
\end{revision}

This distinction is already visible in the time-fractional phase-field
literature, which typically employs the semiconvex Ginzburg--Landau energy. Tang et al.~proved an integral-type energy dissipation law and
numerical stability results for time-fractional phase-field equations, including
Allen--Cahn, Cahn--Hilliard and molecular beam epitaxy models
\cite{TangYuZhou2019}. Quan et al.~constructed a decreasing upper
bound for the energy of the time-fractional Allen--Cahn equation, interpretable
as a nonlocal-in-time modified energy consisting of the original energy plus a
memory accumulation term; they also indicate the applicability of the idea to
other phase-field equations such as Cahn--Hilliard \cite{QuanTangWangYang2022}.
Li et al.~proved fractional energy dissipation laws for nonlocal energy
functionals in a class including Allen--Cahn and Cahn--Hilliard equations
\cite{LiQuanXu2021}. A related augmented-energy viewpoint was developed by
Fritz et al., where the memory effect is represented through
an equivalent integer-order gradient flow in an extended state space
\cite{FritzKhristenkoWohlmuth2023}.

\begin{revision}
There is also a substantial numerical literature for the Allen--Cahn and
Cahn--Hilliard cases. Du et al.~analyzed the time-fractional Allen--Cahn
equation and studied discrete energy dissipation in a weighted average sense
\cite{DuYangZhou2020}. Zhao et al.~studied the coarsening dynamics of
time-fractional phase-field models and reported power-law scaling behavior for
effective free energies in several models, including Cahn--Hilliard-type
dynamics \cite{ZhaoChenWang2018}. Further investigations on the energy in these
time-fractional phase-field models are mentioned in
\cite{LiaoTangZhou2020,ZhangHuangAlikhanovYin2023,ZhangZhaoWang2020,FritzRajendran}.
As another prominent time-fractional gradient flow with an entropy inequality,
we mention the works \cite{FritzSuliWohlmuth,FritzFokker} on the Fokker--Planck
equation with time-fractional derivatives.
\end{revision}

\begin{revision}
At the abstract level, time-fractional gradient flows associated with convex
lower semicontinuous energies were studied in
\cite{Akagi2019,LiSalgado2021,FritzKhristenkoWohlmuth2023}. More recently,
Akagi and Nakajima developed a Hilbert-space theory for nonconvex energies based
on differences of subdifferentials and Lipschitz perturbations
\cite{AkagiNakajima2025}. These works provide the natural functional-analytic
background for the present example.
\end{revision}

The purpose of this short paper is to show that global semiconvexity, even with
smooth compact support, does not force pointwise monotonicity of the original
energy along every exact solution. The counterexample is consistent with the
known dissipation laws cited above: the monotone quantity is an integrated or
memory-modified energy, while the original energy may increase over a finite
interval. The scalar mechanism is the same gap as observed in \cite{Diethelm2016Monotonicity}: the sign
of one fixed Caputo derivative does not determine ordinary monotonicity.

\section{Setup and the question}

\begin{revision}
Let $E\in C^2(\R^d)$ and $0<\alpha<1$. For an absolutely continuous curve
$u\in AC([0,T];\R^d)$, we define its Caputo derivative by
\[
  \CD{\alpha}u:=J^{1-\alpha}[u'],
  \qquad
  J^{\beta}[v](t):=\frac1{\Gamma(\beta)}
  \int_0^t(t-s)^{\beta-1}v(s)\,\dd s,
  \quad 0<\beta<1.
\]
The corresponding classical Caputo equation is
\begin{equation}\label{eq:flow}
  \CD{\alpha}u(t)=-\nabla E\bigl(u(t)\bigr),
  \qquad u(0)=u_0,
  \qquad t\in(0,T].
\end{equation}
We call $u\in C([0,T];\R^d)$ a \emph{mild solution} if it satisfies the
Volterra equation
\begin{equation}\label{eq:mild}
  u(t)=u_0-\frac1{\Gamma(\alpha)}
  \int_0^t(t-s)^{\alpha-1}\nabla E(u(s))\,\dd s,
  \qquad 0\le t\le T.
\end{equation}
If a mild solution is absolutely continuous, then the standard fractional
integral identities
$J^\alpha\CD{\alpha}u=u-u(0)$ and
$\CD{\alpha}J^\alpha f=f$ show that \eqref{eq:mild} and
\eqref{eq:flow} are equivalent almost everywhere. In the explicit construction
below, the Caputo equation holds pointwise for every $t>0$.

In many time-fractional gradient-flow settings, the natural dissipation
statement is not a pointwise differential inequality for the original energy,
but an integrated or memory-modified estimate. In the present paper we do not
use such an abstract estimate as an input. Instead, for the explicitly
constructed solution below we verify directly that
\begin{equation}\label{eq:bounded}
  E(u(t))\le E(u(0))\qquad\text{for }0\le t\le1,
\end{equation}
even though $t\mapsto E(u(t))$ is strictly increasing on a nonempty subinterval.

The pointwise question is whether the ordinary derivative of the original
energy must satisfy
\[
  \frac{\dd}{\dd t}E(u(t))\le0.
\]
If $u\in AC([0,T];\R^d)$ and $E\in C^1(\R^d)$, then $E\circ u$ is absolutely
continuous and the ordinary chain rule gives, for almost every $t\in(0,T)$,
\begin{equation}\label{eq:Edot}
  \frac{\dd}{\dd t}E(u(t))
  =\bigl\langle\nabla E(u(t)),u'(t)\bigr\rangle
  =-\bigl\langle\CD{\alpha}u(t),u'(t)\bigr\rangle
  =-\bigl\langle J^{1-\alpha}[u'](t),u'(t)\bigr\rangle.
\end{equation}
For the explicit curve below, this identity holds pointwise for every $t>0$.
Thus pointwise monotonicity would require pointwise nonnegativity of the memory
form
\[
  Q_\alpha[u'](t):=\bigl\langle J^{1-\alpha}[u'](t),u'(t)\bigr\rangle.
\]
By contrast, for continuous vector-valued functions $v$ the positive-definiteness
of Abel convolution gives the integrated inequality
\[
  \int_0^T\bigl\langle J^{1-\alpha}[v](t),v(t)\bigr\rangle\,\dd t\ge0;
\]
see \cite[Lemma~3.1]{Mustapha2014c}, applied componentwise with convolution
order $1-\alpha$. This background inequality is not used below and is not being
asserted here directly for the explicit singular velocity $u'\in L^1(0,1)$.
The construction exploits precisely the absence of a corresponding pointwise
sign condition. Its scalar mechanism is that the sign of one fixed Caputo
derivative does not determine ordinary monotonicity.
\end{revision}

\section{The energy and its bounded Hessian}
\label{sec:lemma}

\begin{revision}
In the following, we use $\alpha=\frac12$. Let
$\eps\in(0,\tfrac1{2\sqrt2})$ and $M>0$ be arbitrary; the numerical
illustrations use $\eps=0.1$ and $M=1$.
\end{revision}
\rev{We set}
\begin{equation}\label{eq:u}
  u(t)=\begin{pmatrix} \eps\sqrt t \\ \sqrt t\,(1-t)\end{pmatrix}
      =\begin{pmatrix} \eps t^{1/2}\\ t^{1/2}-t^{3/2} \end{pmatrix}.
\end{equation}
\begin{revision}
For $p>0$ and $t>0$, the Caputo power formula
$\CD{1/2}t^{p}=\frac{\Gamma(p+1)}{\Gamma(p+\frac12)}\,t^{p-1/2}$ gives
\end{revision}
\begin{equation}\label{eq:CDu}
  \CD{1/2}u(t)=\begin{pmatrix} \eps\,\Gamma(\tfrac32) \\ \Gamma(\tfrac32)-\Gamma(\tfrac52)\,t \end{pmatrix}
  \qquad \text{ with } \Gamma(\tfrac32)=\tfrac{\sqrt\pi}2,\ \ \Gamma(\tfrac52)=\tfrac32\Gamma(\tfrac32).
\end{equation}
The second component exhibits a sign mismatch as used in \cite{Diethelm2016Monotonicity}: its ordinary
derivative $$\partial_t u_2=\tfrac1{2\sqrt t}(1-3t)$$ changes sign at $t=\tfrac13$,
while $$(\CD{1/2}u)_2=\Gamma(\tfrac32)-\Gamma(\tfrac52)t$$ changes sign at
$$t=\Gamma(\tfrac32)/\Gamma(\tfrac52)=\tfrac23.$$ The mismatch window is therefore
$(\tfrac13,\tfrac23)$.

The crucial structural point is the \emph{clock} $\eps\sqrt t$: its Caputo
derivative is the \emph{constant} $\eps\Gamma(\tfrac32)$. Equivalently, in the
variable
\begin{equation}\label{eq:s}
  s:=\sqrt t,\qquad r(s):=u(s^2)=\begin{pmatrix} \eps s \\ s-s^3\end{pmatrix},
\end{equation}
\begin{revision}
both $r$ and $\CD{1/2}u(s^2)$ are \emph{polynomial in $s$}. Moreover, $r$
is injective because its first component is $\eps s$, and it is a regular
immersion since
\[
  |r'(s)|^2=\eps^2+(1-3s^2)^2\ge\eps^2>0.
\]
\end{revision}

\begin{revision}
We prescribe the gradient along the curve so that \eqref{eq:flow} holds:
\[
\nabla E_0(r(s))=-\CD{1/2}u(s^2)=
\begin{pmatrix}
-\eps\Gamma(\tfrac32) \\
-\Gamma(\tfrac32)+\Gamma(\tfrac52)s^2
\end{pmatrix},
\qquad 0\le s\le1.
\]
\end{revision}
A local realization is the curved valley
\begin{equation}\label{eq:E0}
  E_0(x,y)=B(x)-\Phi(\zeta)\,w+\tfrac{M}{2}\,w^2,
  \qquad
  \zeta=\tfrac{x}{\eps},\quad w=y-\psi(x),
\end{equation}
\begin{revision}
with floor, profile, and base defined by
\begin{equation}\label{eq:pieces}
\begin{aligned}
  \psi(x)&=\zeta-\zeta^3,
  &\Phi(\zeta)&=\Gamma(\tfrac32)-\Gamma(\tfrac52)\zeta^2,\\
  B'(x)&=-\eps\Gamma(\tfrac32)-\Phi(\zeta)\,\psi'(x),
  &B(0)&=0.
\end{aligned}
\end{equation}
Thus $B$ is the uniquely normalized antiderivative, and $M$ is the chosen
transverse stiffness. A direct computation gives the global gradient
\end{revision}
\begin{equation}\label{eq:grad}
  \nabla E_0(x,y)=\begin{pmatrix}
    -\eps\Gamma(\tfrac32)
    -\bigl(\tfrac{\Phi'(\zeta)}{\eps}+M\psi'(x)\bigr)w\\[2pt]
    -\Phi(\zeta)+M w
  \end{pmatrix},
\end{equation}
\begin{revision}
which is polynomial and agrees with the prescribed trace
$-\CD{1/2}u(s^2)$ at $(x,y)=r(s)$. Hence $E_0\in C^\infty(\R^2)$.
\end{revision}
Along the curve \(r(s)\), we have \(x=\eps s\). For notational brevity set
\[
  p(s):=\psi'(\eps s)=\frac{1-3s^2}{\eps},
  \qquad
  q(s):=\Phi'(s)=-2\Gamma(\tfrac52)s .
\]
The Hessian along the curve is therefore
\begin{equation}\label{eq:hess}
  \Hess E_0\bigl(r(s)\bigr)
  =
  \begin{pmatrix}
    \dfrac{q(s)p(s)}{\eps}+M p(s)^2
    &\quad 
    -\dfrac{q(s)}{\eps}-M p(s)\\[6pt]
    -\dfrac{q(s)}{\eps}-M p(s)
    &
    M
  \end{pmatrix}.
\end{equation}

\begin{lemma}[No anchor obstruction for $\sqrt t$-clocks]\label{lem:noblow}
\begin{revision}
Along $r$, the tangential quadratic form---which is determined by the
prescribed gradient trace along the curve---is bounded and satisfies
\begin{equation}\label{eq:tang}
\begin{aligned}
  r'(s)^{\!\top}\Hess E_0(r(s))\,r'(s)
  &=\frac{\dd}{\dd s}\bigl[\nabla E_0(r(s))\bigr]\cdot r'(s)\\
  &=2\,\Gamma(\tfrac52)\,s\,(1-3s^2),
  \qquad s\in[0,1].
\end{aligned}
\end{equation}
Moreover, the full Hessian matrix $\Hess E_0(r(s))$ has bounded entries on
$[0,1]$. Consequently,
\[
  \inf_{s\in[0,1]}
  \lambda_{\min}\!\bigl(\Hess E_0(r(s))\bigr)>-\infty.
\]
\end{revision}
\end{lemma}

\begin{proof}
The identity \eqref{eq:tang} follows from
\[
  \frac{\dd}{\dd s}\nabla E_0(r(s))
  =
  \Hess E_0(r(s))\,r'(s),
\]
with $r'(s)=(\eps,1-3s^2)$ and
\[
  \frac{\dd}{\dd s}\nabla E_0(r(s))
  =
  \begin{pmatrix}0\\ 2\Gamma(\tfrac52)s\end{pmatrix}.
\]
Since $p(s)$ and $q(s)$ are polynomial functions of $s$, the displayed formula
\eqref{eq:hess} shows that every entry of $\Hess E_0(r(s))$ is bounded on
$[0,1]$. Hence the eigenvalues, being continuous functions of the matrix
entries, are bounded on $[0,1]$.
\begin{revision}
At $s=0$ one has
\[
  \Hess E_0(r(0))
  =
  M\begin{pmatrix}\eps^{-2}&-\eps^{-1}\\-\eps^{-1}&1\end{pmatrix},
\]
which is positive semidefinite with eigenvalues $0$ and $M(1+\eps^{-2})$.
\end{revision}
\end{proof}

\begin{theorem}[Strict increase of the original energy]\label{thm:increase}
\begin{revision}
With $E_0$ as in \eqref{eq:E0} and $u$ as in \eqref{eq:u}, for every
$t\in(0,1]$,
\begin{equation}\label{eq:Edotclosed}
  \frac{\dd}{\dd t}E_0(u(t))
  =-\bigl\langle\CD{1/2}u(t),u'(t)\bigr\rangle
  =\frac{\Gamma(\tfrac32)}{2\sqrt t}
   \Bigl[\tfrac92\,t(1-t)-(1+\eps^2)\Bigr].
\end{equation}
Hence $\frac{\dd}{\dd t}E_0(u(t))>0$ precisely on the interval
\begin{equation}\label{eq:window}
  t\in(t_-,t_+),\qquad
  t_\pm=\frac12\Bigl(1\pm\sqrt{1-\tfrac{8(1+\eps^2)}9}\Bigr),
\end{equation}
which is nonempty if and only if $\eps<\tfrac1{2\sqrt2}$. The energy derivative
attains its global maximum on $(0,1]$ at
\begin{equation}\label{eq:tstar}
  t_*=\frac{9+\sqrt{81+216(1+\eps^2)}}{54}.
\end{equation}
For $\eps=0.1$,
\[
  (t_-,t_+)\approx(0.34013895,\,0.65986105),\qquad
  t_*\approx0.48696745,
\]
and
\[
  \max_{0<t\le1}\frac{\dd}{\dd t}E_0(u(t))
  \approx0.07253820.
\]
As $\eps\downarrow0$, the endpoints $t_-$ and $t_+$ converge to
$\tfrac13$ and $\tfrac23$, respectively; the construction itself always assumes
$\eps>0$.
\end{revision}
\end{theorem}

\begin{proof}
Insert \eqref{eq:CDu} and \[\partial_t u=\begin{pmatrix}\tfrac{\eps}{2\sqrt t} \\ \tfrac1{2\sqrt t}-\tfrac{3\sqrt t}2\end{pmatrix},\] into \eqref{eq:Edot}, factor
$\tfrac1{2\sqrt t}$, and use $\Gamma(\tfrac52)=\tfrac32\Gamma(\tfrac32)$ to get
\eqref{eq:Edotclosed}. The sign is that of
$\tfrac92 t(1-t)-(1+\eps^2)$, and
$t(1-t)>\tfrac{2(1+\eps^2)}9$ gives \eqref{eq:window}.
\begin{revision}
Writing the right-hand side of \eqref{eq:Edotclosed} as $F(t)$ and
differentiating gives
\[
  F'(t)=\frac{\Gamma(\tfrac32)}{8t^{3/2}}
  \bigl[2(1+\eps^2)+9t-27t^2\bigr].
\]
Thus the critical points satisfy
$27t^2-9t-2(1+\eps^2)=0$. This equation has exactly one positive root, namely
\eqref{eq:tstar}; moreover, $F'(t)>0$ before this root and $F'(t)<0$ after it.
Since $F(t)\to-\infty$ as $t\downarrow0$, $t_*$ is the global maximizer on
$(0,1]$.
\end{revision}
\end{proof}

Moreover, along the constructed trajectory one can integrate
\eqref{eq:Edotclosed} explicitly:
\[
  E_0(u(t))-E_0(u(0))
  =
  \Gamma(\tfrac32)\sqrt t
  \left[
    -(1+\eps^2)+\frac32 t-\frac9{10}t^2
  \right].
\]
Since the quadratic polynomial
\[
  \frac32 t-\frac9{10}t^2
\]
attains its maximum \(5/8\) on \([0,1]\), we have
\[
  E_0(u(t))-E_0(u(0))
  \le
  \Gamma(\tfrac32)\sqrt t
  \left[-(1+\eps^2)+\frac58\right]
  <0
\]
for every \(t\in(0,1]\). Thus the energy remains below its initial value,
even though it is strictly increasing on \((t_-,t_+)\).

\section{Well-posedness}

\begin{revision}
We now verify well-posedness in the solution class introduced in
Section~2. With $\alpha=\tfrac12$ and $u_0=0$, the mild formulation
\eqref{eq:mild} is
\begin{equation}\label{eq:volterra}
  u(t)=-\frac1{\Gamma(\tfrac12)}
  \int_0^t(t-s)^{-1/2}\nabla E(u(s))\,\dd s,
  \qquad 0\le t\le1.
\end{equation}
More generally, suppose that $\nabla E$ is globally Lipschitz with constant
$L$ on $\R^d$, and fix a finite $T>0$. On $C([0,T];\R^d)$ define
\[
  (\mathcal Tv)(t):=u_0-\frac1{\Gamma(\alpha)}
  \int_0^t(t-s)^{\alpha-1}\nabla E(v(s))\,\dd s.
\]
For the weighted norm
$\|v\|_\beta:=\sup_{0\le t\le T}e^{-\beta t}|v(t)|$, one has
\[
  \|\mathcal Tv-\mathcal Tw\|_\beta
  \le L\beta^{-\alpha}\|v-w\|_\beta,
\]
because
$\Gamma(\alpha)^{-1}\int_0^\infty r^{\alpha-1}e^{-\beta r}\,\dd r
=\beta^{-\alpha}$. Choosing $\beta>L^{1/\alpha}$ makes $\mathcal T$ a
contraction. Hence the Volterra equation has a unique continuous mild solution
on every prescribed finite interval.

The curve $u$ from \eqref{eq:u} belongs to $AC([0,1];\R^2)$ because
$u'\in L^1(0,1;\R^2)$, and it satisfies
$\CD{1/2}u=-\nabla E_0(u)$ pointwise on $(0,1]$. After the global cutoff below,
the same identity holds with $E$ in place of $E_0$. Consequently the explicit
curve is a mild solution of \eqref{eq:volterra}; uniqueness then identifies it
with the unique mild solution. It is also the unique absolutely continuous
classical Caputo solution on $[0,1]$.
\end{revision}

\begin{proposition}[Global semiconvex realization on the counterexample interval]
\label{prop:global}
\begin{revision}
For every $\eps\in(0,\tfrac1{2\sqrt2})$ and every $M>0$, there exist
$E\in C_c^\infty(\R^2)$ and $\Lambda<\infty$ such that
\[
  \Hess E\ge-\Lambda I\qquad\text{on }\R^2,
\]
and the curve $u$ from \eqref{eq:u} is the unique continuous mild solution on
$[0,1]$ of \eqref{eq:volterra}. Since $u\in AC([0,1];\R^2)$, it is also the
unique absolutely continuous classical solution satisfying
\[
  \CD{1/2}u(t)=-\nabla E(u(t)),\qquad u(0)=0,
  \qquad 0<t\le1.
\]
Moreover, all conclusions of Theorem~\ref{thm:increase} hold with $E$ in place
of $E_0$ on $[0,1]$.
\end{revision}
\end{proposition}

\begin{proof}
Choose $\rho>0$ and define the closed tube
\[
  K_\rho
  :=
  \{(x,y):0\le x\le\eps,\ |y-\psi(x)|\le \rho\}.
\]
Then $r([0,1])\subset K_\rho$. Choose bounded open sets
$V_\rho,U_\rho\subset\R^2$ such that
\[
  K_\rho\Subset V_\rho\Subset U_\rho .
\]
Let \(\chi\in C_c^\infty(\R^2)\) satisfy
\[
  0\le \chi\le 1,\qquad
  \chi\equiv1 \text{ on } V_\rho,\qquad
  \supp\chi\subset U_\rho .
\]
Set
\[
  E:=\chi E_0 .
\]
Then $E\in C_c^\infty(\R^2)$. Since $\Hess E$ is continuous and compactly
supported, it is bounded. Hence, with
\begin{revision}
\[
  \Lambda:=\sup_{z\in\R^2}\|\Hess E(z)\|_{\mathrm{op}}<\infty,
\]
\end{revision}
we have
\[
  \Hess E\ge -\Lambda I \qquad\text{on }\R^2.
\]
Thus $E$ is globally semiconvex.

For $0\le t\le1$, one has $u(t)=r(\sqrt t)\in K_\rho\subset V_\rho$. Since
$\chi\equiv1$ on $V_\rho$, it follows that
\[
  \nabla E(u(t))=\nabla E_0(u(t))
  \qquad\text{for }0\le t\le1.
\]
The explicit construction gives
\[
  \CD{1/2}u(t)=-\nabla E_0(u(t))
  \qquad\text{for }0<t\le1,
\]
and therefore
\[
  \CD{1/2}u(t)=-\nabla E(u(t))
  \qquad\text{for }0<t\le1.
\]
Equivalently, $u$ satisfies \eqref{eq:volterra} on $[0,1]$.
\begin{revision}
Finally, since $E\in C_c^\infty(\R^2)$, its Hessian is globally bounded and
$\nabla E$ is globally Lipschitz. The contraction argument above therefore
implies that the explicitly constructed $u$ is the unique continuous mild
solution on $[0,1]$, and hence also the unique absolutely continuous classical
solution. The energy-increase computation is unchanged because $E=E_0$ on a
neighborhood of $u([0,1])$.
\end{revision}
\end{proof}

\begin{remark}[Conventions for semiconvexity]
We use ``globally semiconvex'' to mean that there exists $\Lambda\ge0$ such
that
\(
  \Hess E\ge -\Lambda I.
\)
Equivalently, $E+\frac{\Lambda}{2}|\cdot|^2$ is convex. Under the signed
``$\lambda$-convexity'' convention this is the same as $(-\Lambda)$-convexity.
\end{remark}

\begin{remark}[Coercivity]
\begin{revision}
The compactly supported energy $E=\chi E_0$ is bounded below. A coercive variant
can be written explicitly. Choose $\vartheta\in C_c^\infty(\R^2)$ with
$\vartheta\equiv1$ on a neighborhood of $K_\rho$, and for $\delta>0$ set
\[
  \widetilde E(z):=\vartheta(z)E_0(z)
  +(1-\vartheta(z))\frac{\delta}{2}|z|^2.
\]
Then $\widetilde E=E_0$ near the trajectory and
$\widetilde E(z)=\frac\delta2|z|^2$ outside a compact set, so $\widetilde E$ is
coercive. Its Hessian is bounded below on the compact transition region and
equals $\delta I$ outside that region; hence $\widetilde E$ is globally
semiconvex and leaves the constructed trajectory unchanged.
\end{revision}
\end{remark}

\begin{remark}[On the value of the global semiconvexity modulus]
\begin{revision}
The lemma gives a finite lower Hessian bound along the curve. Since
$E_0\in C^\infty(\R^2)$, compactness also gives, for every fixed $\rho>0$,
\[
  -\inf_{K_\rho}\lambda_{\min}(\Hess E_0)<\infty.
\]
The numerical value of an on-tube or global semiconvexity modulus depends on
$\eps$, $M$, the tube radius, and the chosen cutoff. In particular, derivatives
of the cutoff can make the global modulus substantially larger than the
on-curve lower bound. No quantitative estimate of this inflation is used in the
argument; only finiteness of the modulus is required.
\end{revision}
\end{remark}

\begin{remark}[Heuristic instability mechanism]
The following observation is heuristic and is not used in the proof of the
counterexample. Linearizing \eqref{eq:flow} about the constructed solution gives
formally
\begin{equation}\label{eq:lin}
  \CD{1/2}\delta=-H(t)\delta,\qquad
  H(t)=\Hess E(u(t)).
\end{equation}
Along the constructed trajectory the Hessian $H(t)$ has negative directions for
part of the interval. Frozen-coefficient scalar components with negative
curvature lead to equations of the form
\[
  \CD{1/2}\eta=\mu\eta,\qquad \mu>0,
\]
whose solutions are
\[
  \eta(t)=\eta_0 E_{1/2}(\mu\sqrt t).
\]
\begin{revision}
Here $E_\alpha(z):=\sum_{k=0}^\infty z^k/\Gamma(\alpha k+1)$ denotes the
one-parameter Mittag--Leffler function.
\end{revision}
This suggests a possible amplification mechanism for perturbations of the
energy-increasing trajectory. Since $H(t)$ is time dependent and its eigenspaces
need not commute in time, this frozen-coefficient comparison does not by itself
prove nonlinear or even linear instability of the full time-dependent problem.
\end{remark}

\section{Numerical validation}\label{sec:numval}

\begin{revision}
The construction is validated directly against the closed forms of
Section~\ref{sec:lemma} and the energy identity \eqref{eq:Edotclosed}. For
reproducibility, let
\[
  \mathcal G:=\{j/10^4:j=1,\ldots,10^4\}.
\]
The gradient-trace residual
\[
  R_{\mathrm{tr}}
  :=\max_{t\in\mathcal G}
  \bigl\|\nabla E_0(u(t))+\CD{1/2}u(t)\bigr\|
\]
reported in Table~\ref{tab:checks} was evaluated in IEEE double precision on
$\mathcal G$. The Hessian minimum was obtained by bounded
one-dimensional minimization of the exact smallest-eigenvalue formula, with
absolute tolerance $10^{-12}$. These numerical checks are not used as premises
in the analytic proof.
\end{revision}

\begin{table}[htp!]
\begin{revision}
\centering
\caption{Validation of the defining properties ($\eps=0.1$, $M=1$).}
\label{tab:checks}
\begin{tabular}{p{0.57\linewidth}p{0.33\linewidth}}
\hline\noalign{\smallskip}
Quantity & Value\\
\noalign{\smallskip}\hline\noalign{\smallskip}
Gradient-trace residual $R_{\mathrm{tr}}$
& $4.44\times10^{-16}$\\
Energy-increase window $(t_-,t_+)$
& $(0.34013895,\,0.65986105)$\\
Maximizer $t_*$ of $\frac{\dd}{\dd t}E_0(u(t))$
& $0.48696745$\\
Maximum energy derivative
& $0.07253820$\\
Hessian minimum
$\min_{0\le s\le1}\lambda_{\min}(\Hess E_0(r(s)))$
& $-33.48349$  at $s=0.456109$\\
\noalign{\smallskip}\hline
\end{tabular}
\end{revision}
\end{table}

\begin{revision}
Figure~\ref{fig:energy} confirms the strict energy increase. As $t\downarrow0$,
$\frac{\dd}{\dd t}E_0(u(t))\to-\infty$; the left panel therefore starts the
curve at $t=0.02$ and clips values below $-0.2$. The derivative crosses zero at
$t_-$, is positive on $(t_-,t_+)$, reaches its maximum
$0.07253820$ at $t_*\approx0.48696745$, and becomes negative again at $t_+$.
The dashed lines at $t=\tfrac13$ and $t=\tfrac23$ mark, respectively, the sign
changes of $u_2'(t)$ and $(\CD{1/2}u(t))_2$; they are not zeros of the energy
derivative. The right panel shows that the shifted energy
$E_0(u(t))-E_0(u(0))$ rises by approximately $1.548\times10^{-2}$ on the
shaded window, while remaining below its initial value in accordance with
\eqref{eq:bounded}.
\end{revision}

\begin{figure}[htp!]
\centering
\begin{minipage}{0.48\linewidth}\centering
\begin{tikzpicture}
\begin{axis}[width=0.90\linewidth,height=4.8cm,
  xlabel={$t$},ylabel={$\tfrac{\dd}{\dd t}E_0(u(t))$},
  xmin=0,xmax=1,ymin=-0.2,ymax=0.1,
  axis lines=left,tick align=outside,
  every axis title/.style={below right,at={(0,1)}}]
  \fill[blue!12] (axis cs:0.34014,-0.2) rectangle (axis cs:0.65986,0.1);
  \draw[densely dashed,black!55] (axis cs:0.33333,-0.2)--(axis cs:0.33333,0.1);
  \draw[densely dashed,black!55] (axis cs:0.66667,-0.2)--(axis cs:0.66667,0.1);
  \draw[black!45] (axis cs:0,0)--(axis cs:1,0);
  \addplot[blue,thick,domain=0.02:1,samples=200]
    {0.886226925452758/(2*sqrt(x))*(4.5*x*(1-x)-1.01)};
\end{axis}
\end{tikzpicture}
\end{minipage}\hfill
\begin{minipage}{0.48\linewidth}\centering
\begin{tikzpicture}
\begin{axis}[width=0.90\linewidth,height=4.8cm,
  xlabel={$t$},ylabel={$E_0(u(t))-E_0(u(0))$},
  xmin=0.25,xmax=0.75,ymin=-0.315,ymax=-0.293,
  axis lines=left,tick align=outside,
  every axis title/.style={below right,at={(0,1)}},
  yticklabel style={/pgf/number format/fixed,/pgf/number format/precision=3}]
  \fill[blue!12] (axis cs:0.34014,-0.315) rectangle (axis cs:0.65986,-0.293);
  \addplot[red,thick,domain=0.25:0.75,samples=200]
    {0.886226925452758*sqrt(x)*(-1.01+1.5*x-0.9*x^2)};
\end{axis}
\end{tikzpicture}
\end{minipage}
\caption{\rev{Energy increase for $\eps=0.1$. \emph{Left:} the ordinary energy
derivative, plotted for $t\ge0.02$ and vertically clipped below $-0.2$, is
positive precisely on the shaded window $(t_-,t_+)$. The dashed lines mark the
separate sign changes of $u_2'$ and $(\CD{1/2}u)_2$ at $t=\tfrac13$ and
$t=\tfrac23$. \emph{Right:} the shifted energy is strictly increasing on the
same window while remaining below its initial value.}}
\label{fig:energy}
\end{figure}
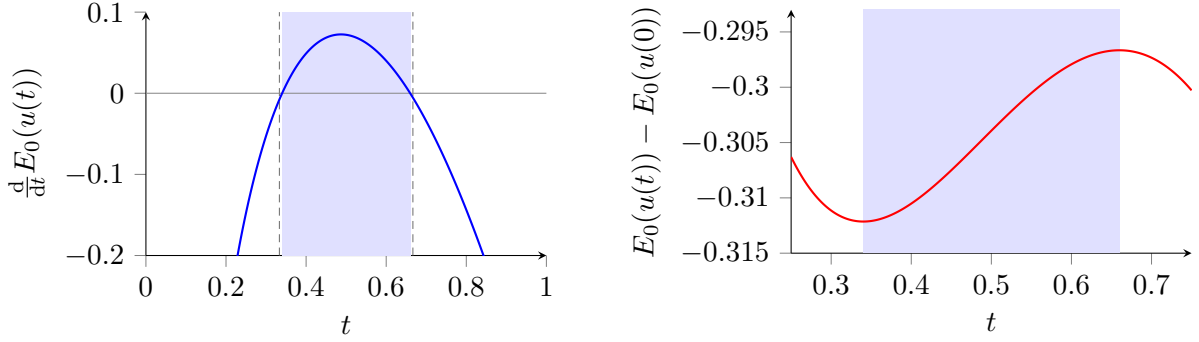

\begin{revision}
Figure~\ref{fig:hess} shows the smallest eigenvalue of $\Hess E_0$ along the
curve $r(s)$ for $\eps=0.1$ and $M=1$. It vanishes at the anchor $s=0$, reaches
approximately $-33.48349$ at $s\approx0.456109$, and stays bounded for all
$s\in[0,1]$, in agreement with Lemma~\ref{lem:noblow}.
\end{revision}

\begin{figure}[htp!]
\centering
\begin{tikzpicture}
\begin{axis}[width=0.62\linewidth,height=4.8cm,
  xlabel={$s=\sqrt t$},ylabel={$\lambda_{\min}(\Hess E_0(r(s)))$},
  xmin=0,xmax=1,axis lines=left,tick align=outside]
  \draw[black!45] (axis cs:0,0)--(axis cs:1,0);
  \addplot[blue,thick,domain=0:1,samples=250]
    {(
      ((-2.65868077635827*x)*(10*(1-3*x^2))/0.1 + (10*(1-3*x^2))^2) + 1
      - sqrt(((-2.65868077635827*x)*(10*(1-3*x^2))/0.1 + (10*(1-3*x^2))^2 - 1)^2
      + 4*(-(-2.65868077635827*x)/0.1 - 10*(1-3*x^2))^2)
    )/2};
\end{axis}
\end{tikzpicture}
\caption{\rev{Smallest eigenvalue of $\Hess E_0$ along $r(s)$ for
$\eps=0.1$ and $M=1$: finite and bounded below, with minimum approximately
$-33.48349$ at $s\approx0.456109$ and no anchor blow-up.}}
\label{fig:hess}
\end{figure}
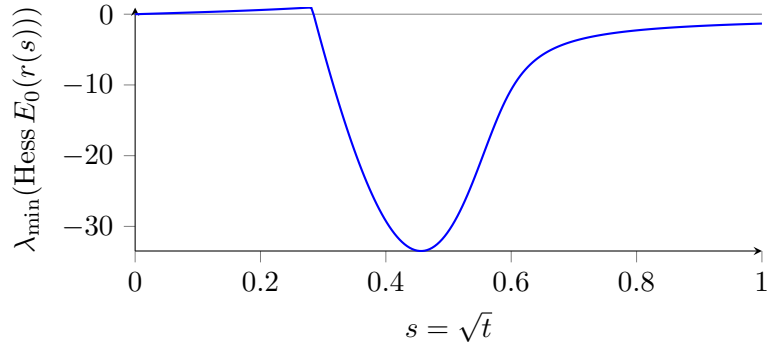

\section{Conclusion}

The proposed counterexample refutes the implication
\[
  \text{finite semiconvexity modulus}
  \quad\Longrightarrow\quad
  \frac{\dd}{\dd t}E(u(t))\le0
  \text{ along every solution}.
\]
The counterexample is constructed on the finite interval $[0,1]$, which is
sufficient because the energy-increase window is a nonempty subinterval of
$(0,1)$.

\begin{revision}
The convex case appears more rigid, and monotonicity can be verified directly
for convex quadratic energies. Let
\[
  E(x)=\frac12x^\top Ax+b^\top x+c,
  \qquad A=A^\top\succeq0.
\]
After orthogonally diagonalizing $A$, write the corresponding coordinates as
$y_i$ and the components of the transformed vector $b$ as $b_i$. For an
eigenvalue $\lambda_i>0$, the shifted variable
$z_i:=y_i+b_i/\lambda_i$ satisfies
\[
  z_i(t)=z_i(0)E_\alpha(-\lambda_i t^\alpha),
\]
and the associated energy contribution is
$\frac12\lambda_i z_i(t)^2-\frac{b_i^2}{2\lambda_i}$. The function
$x\mapsto E_\alpha(-x)$ is nonnegative and nonincreasing for
$0<\alpha<1$ \cite{Pollard1948}, so this contribution is nonincreasing. If
$\lambda_i=0$, then
\[
  y_i(t)=y_i(0)-\frac{b_i}{\Gamma(1+\alpha)}t^\alpha,
\]
and the affine energy contribution $b_i y_i(t)$ is again nonincreasing. Summing
over the modes proves monotonicity of the original energy for every convex
quadratic $E$.

This motivates the following finite-dimensional smooth conjecture.
\begin{quote}\itshape
Let $E\in C^1(\R^d)$ be convex, and assume that $\nabla E$ is globally
Lipschitz. Let $u$ be the unique continuous mild solution on $[0,T]$, in the sense that
\[
\begin{aligned}
  u(t)&=u_0-\frac1{\Gamma(\alpha)}
  \int_0^t(t-s)^{\alpha-1}\nabla E(u(s))\,\dd s,\\
  &\hspace{7em}0\le t\le T,\qquad 0<\alpha<1.
\end{aligned}
\]
Then $t\mapsto E(u(t))$ is nonincreasing on $[0,T]$.
\end{quote}
Theorem~\ref{thm:increase} shows that such a statement cannot be extended to
the class of globally semiconvex, possibly nonconvex energies. A separate
question, not addressed here, is whether pointwise monotonicity can be recovered
for certain nonconvex energies under additional dynamical stability assumptions,
such as a Mittag--Leffler spectral condition or a fractional Lyapunov estimate.
\end{revision}


\end{document}